\documentclass[11pt]{amsart}
\usepackage[left=1.5in, right=1.5in]{geometry}
\usepackage{amssymb}

\usepackage{amscd}

\usepackage{graphicx,psfrag,epsfig,multirow}
\usepackage{amssymb,amsmath,amscd,amsthm,verbatim}
\numberwithin{equation}{section}

\usepackage[active]{srcltx}
\usepackage{geometry}
\usepackage{amssymb,amsmath,amscd,amsthm}
\usepackage{graphicx}

\newtheorem{Thm}{Theorem}
\newtheorem{Def}{Definition}
\newtheorem{Lm}{Lemma}
\newtheorem{Prop}[Lm]{Proposition}
\newtheorem{Rk}{Remark}
\newtheorem{Cor}[Lm]{Corollary}

\def\bdef{\begin{Def}}
\def\endef{\end{Def}}
\def\bthm{\begin{Thm}}
\def\ethm{\end{Thm}}
\def\bprop{\begin{Prop}}
\def\enprop{\end{Prop}}
\def\blm{\begin{Lm}}
\def\elm{\end{Lm}}
\def\bcor{\begin{Cor}}
\def\ecor{\end{Cor}}
\def\brm{\begin{Rem}}
\def\erm{\end{Rem}}
\def\bfig{\begin{picture}}
\def\efig{\end{picture}}
\def\be{\begin{eqnarray}}
\def\ee{\end{eqnarray}}
\def\beal{\begin{aligned}}
\def\enal{\end{aligned}}

\def\Z{\mathbb Z}
\def\T{\mathbb T}

\def\R{\mathbb R}
\def\eps{\varepsilon}
\def\al{\alpha}
\def\bt{\beta}
\def\dt{\delta}
\def\gm{\gamma}

\def\Om{\Omega}

\def\pt{\partial}

\title{Arnold Diffusion in a Restricted Planar Four-Body Problem}
\author{Jinxin Xue\\ Department of Mathematics,
 University of Maryland, \\
College Park, MD, 20742}
\begin{document}
\thanks{\emph{Email:} jxue@math.umd.edu}%
\maketitle

\begin{abstract}
In this paper, we construct a certain planar four-body problem which exhibits fast energy growth under certain assumption.  The system considered is a quasi-periodic perturbation of the Restricted Planar Circular three-body Problem (RPC3BP).  Gelfreich-Turaev's and de la Llave's mechanism is employed to obtain the fast energy growth.  The diffusion is created by a heteroclinic cycle formed by two Lyapunov periodic orbits surrounding $L_1$ and $L_2$ Lagrangian points and their heteroclinic intersections. Our model is the first known example in celestial mechanics about the a priori chaotic case of Arnold diffusion \cite{1}.
\end{abstract}
\section{Introduction}\label{section: intro}
In this paper, we construct a model of the restricted planar four-body problem (RP4BP) which exhibit long time instabilities, i.e. motions change substantially under certain assumption.  The model we employ can be considered as a Sun-Jupiter-Planet(small)-Asteroid system.  In this model the mass of the asteroid is assumed to be negligibly small (in fact zero). The mass of the planet, denoted by $\dt$, is strictly positive and $0 < \dt \ll \mu$ where $\mu$ denotes the mass of Jupiter and the mass of the sun is set to be $1-\mu$.  The Sun-Jupiter-Planet (S-J-P) system forms a planar three-body problem (P3BP) which has quasi-periodic motion, and the objective of study is to understand the motions of the massless asteroid in this system.
Gelfreich and Turaev proposed the following mechanism for Arnold
diffusion \cite{2}. (It was pointed out to me that R. de la Llave proposed the same mechanism in his unpublished paper \cite{3} and mentioned it in his ICM 2006 talk. So in the following, we call it the $GTL$ mechanism.) For a nonautonomous Hamiltonian system,
$H(p,q,\eps t),\ q \in \T^n,\ p\in \R^n$, consider the frozen system, $H(p,q,\nu)$.
Suppose for each $\nu$, and each energy surface of the energy interval $[h_-,h_+]$ of the
frozen Hamiltonian $H(p,q,\nu)$, there are hyperbolic periodic orbits $\gm_1$ and $\gm_2$ with stable and unstable manifolds $W^{u,s}(\gm_1)$, $W^{u,s}(\gm_2)$ which make transversal heteroclinic intersections. In \cite{2} it is proven that,
for a sufficiently small $\eps$, there exists $t_1>0$ such
that the Hamiltonian $H(p,q,\eps t)$ has a linearly fast diffusing
orbit $(p,q)(t)$, i.e.
\[
H(p(0),q(0),0)=h_-, \quad H(p(t),q(t),\eps t)=h_+
\]
for some $t\le t_1/\eps$.\\
One important feature of the mechanism of \cite{2} is that their methods do not rely on how the Hamiltonian depends on $\eps t$. The results hold true for periodic, quasiperiodic and other settings. The key to this paper is to apply the \cite{2} mechanism to the RP4BP. In our case, the $\eps t$ dependence is quasi-periodic.

Notice that in restricted planar circular three-body problem (RPC3BP) there are two normally
hyperbolic periodic orbits $\gm_1,\gm_2$ surrounding the $L_1$ and $L_2$ Lagrangian points respectively. We will have the ``heteroclinic cycle" required in \cite{2} if we know that their stable and unstable manifolds have heteroclinic intersections.

To obtain slow time-quasi-periodic perturbation we need to exploit the
planet. We select S-J-P to have quasi-periodic orbit, along which the planet is far from the sun and Jupiter and the distance is of order $\eps^{-2/3}$, so the motions of the asteroid can be described as solutions to the Hamiltonian system with a Hamiltonian $H_{A,rot}$ that is a small and slow quasi-periodic perturbation of the RCP3BP,
 (see equation ~(\ref{eq: main})): \[H_{A,rot}(L_A,\ell_A, G_A, g_A, t)=\mathrm{RPC3BP}+ f(L_A,\ell_A, G_A, g_A,\eps t),\quad (L_A,\ell_A, G_A, g_A)\in T^*\T^2,\]
where the function $f=o(1)$ as $\dt\to 0$ is complicated, whose complete expression will be given in Theorem \ref{Thm: ham}. The variables that we are using here are called Delaunay coordinates (see Section 3 and Appendix \ref{AppendixDelaunay} for more details).

The main result proved in this paper is:
\begin{Thm}
If in RPC3BP, for some fixed $\mu>0$ sufficiently small, there exists an energy interval $(\hat h_-,\ \hat h_+)$ such that in each energy level $h\in (\hat h_-,\ \hat h_+)$, there are two Lyapunov periodic orbits $\gm_1,\gm_2$ whose diameters are sufficiently small, and their stable and unstable manifolds intersect transversally. Then in the Sun-Jupiter-Planet-Asteroid system $H_{A,rot}$ for the same $\mu$ and for  $\eps$ and $\dt>0$ small enough satisfying $\dt=O(\eps^{3})$, there exists a diffusion orbit $(L_A,\ell_A,G_A,g_A)(t)$ with linearly fast energy growth. i.e. There is an energy interval $[h_-,h_+]\subset (\hat h_-,\ \hat h_+)$ independent of $\eps,\delta$, such that the energy $H_{A,rot}$ of the asteroid  has  growth:
\[H_{A,rot}((L_A,\ell_A,G_A,g_A)(0),0)=h_-, \quad H_{A,rot}((L_A,\ell_A,G_A,g_A)(t),\eps t)=h_+\]
for some $t\leq \mathrm{const.}/(\dt\eps^{1/3})$, where $h_\pm$ satisfies $|\hat h_+-h_+|,\ |\hat h_--h_-|=o(1)$ as  $\dt\to 0$.\label{Thm: main}
\end{Thm}

\begin{Rk}
The assumption in the theorem on the existence of Lyapunov periodic orbits is known to be true $($see Section \ref{section: hyp}$)$, but the part on the transversal intersection of stable and unstable manifolds of the two Lyapunov periodic orbits remains an open problem. In the worst case when energy interval is empty, the theorem is void and we do not have energy growth. However, there are several numerical results in support of the assumption.
\begin{itemize}
\item
In \cite{5}, the authors did rigorous numerics to show the ``topological" intersection of the stable and unstable manifolds of the two Lyapunov periodic orbits for realistic mass ratio $\mu$ and energy level and constructed symbolic dynamics using the periodic orbits and the topological intersections.
\item There is numerical evidences in \cite{12} confirming the assumption and showing that this energy interval is ``wide".
\item Moreover, in the case of $\mu$ sufficiently small, there is also numeric evidence in \cite{18} supporting the assumption for the Hill problem, a limiting case of RPC3BP. In that case, we expect to have an estimate of the size of the energy interval $O(\mu^{1/3}),\ \mu\to 0$. Once $\mu$ is chosen and fixed, this size of interval is independent of $\dt,\eps$.
\end{itemize}
We will talk about these results in more details in Section \ref{subsection: hetero}. Our theorem is stated to be compatible with the third bullet point. In the proof, we also consider the case $\mu$ not small. In that case, we also get energy growth if we have transversal intersections of the stable and unstable manifolds of the Lyapunov orbits and the issue raised in later Remark \ref{RkFinal} can be checked numerically.
\end{Rk}
Our diffusion orbit has the following behavior. For most of the time the asteroid follows one of the Lyapunov orbit and gains energy growth. When the energy ceases to grow, the asteroid jumps to the other Lyapunov orbit along the heteroclinic orbit and gains energy growth again. This procedure is repeated until we lose the heteroclinic cycle structure up to some energy level.

The problem of the Arnold diffusion is a long story concerning the instability of generic Hamiltonian systems.
Here we do not try to mention the seas of literatures about Arnold diffusion, but point out the results relevant to our work. Even though the problem of Arnold diffusion has been studied for half a century, there are scarce concrete examples, esp. in celestial mechanics. As far as the author's knowledge, the only known examples are \cite{14,15,16,17}. Their mechanisms are all that of Arnold's original mechanism, called a priori unstable case.

However, our model has new features. The study of energy growth is a simplified version of Arnold diffusion by Mather, so it is also called the Mather problem \cite{1}. The mechanism of diffusion is called ``a priori chaotic" in \cite{1}, since  the reference system has some conserved quantities, but there are orbits which are hyperbolic and with transverse heteroclinic intersections in the manifolds corresponding to the conserved quantities. The systems are not close to integrable, so the Nekhoroshev upper bounds for the time of diffusion does not apply. Our model is the first known model about the a priori chaotic case in celestial mechanics.

Moreover, in fluid mechanics there is a phenomena of secondary flow discovered by Prandtl, which is
in general produced when the centripetal force does not balance the pressure. As an example, consider water circulating in a bowl or cup, the primary flow is circular and might be expected to push heavy particles outward to the perimeter. However, heavy particles such as tea leaves in fact congregate in the center as a result of the secondary flow.  The secondary flow is related to the double exponential growth of vorticity in some models (see \cite{19} for a recent result and references therein). In RPC3BP, the two Lyapunov periodic orbits can be seen as an analogue of the secondary flow since they are also created from the competition of the centripetal force and the force coming from the two primaries. So we expect the energy growth studied in our model may be related to some unstable behaviors (e.g. double exponential growth, or blowup) in fluid mechanics and provide new insights if possible.

Since the RPC3BP is an autonomous system, we do not expect any energy growth in it. Even though we impose assumption on the existence of transversal intersection of stable and unstable manifolds of the Lyapunov periodic orbits, there are still many things to do in order to show our restricted four-body problem has energy growth. The paper is organized as follows. In Section \ref{section: intro2GT}, we give a brief introduction to the GTL mechanism following Gelfreich-Turaev \cite{2}.

In Section \ref{section: per}, we give the construction of the configuration of the four-body problem.  We first find a quasi-periodic orbits for full S-J-P three-body problem. Then we put a massless asteroid into the system and write down the Hamiltonian governing the motion of the asteroid as a quasi-periodic perturbation of the RPC3BP (S-J-A).

In Section \ref{section: SJA}, the RPC3BP(S-J-A) is studied. There are two normally hyperbolic periodic orbits $\gm_1(h),\gm_2(h)$ around the $L_1$ and $L_2$ Lagrangian points respectively, on each energy level $h$ for an energy interval $h\in [h_{-},h_{+}]$ (\cite{4}). We also mention the known numerical results supporting the transversal heteroclinic intersections in \cite{5} and \cite{18,12}.

In Section \ref{section: hyp}, the heteroclinic cycle of the RPC3BP (S-J-A) is transplanted to the RP4BP. This is done using the hyperbolic theory.

In the last section \ref{section: nondegenerate}, GTL mechanism is applied to the RP4BP. We check the nondegeneracy condition required in \cite{2} in this section.

Finally, we have two appendices to introduce the Delaunay and polar coordinates of two-body problem.
\section{A brief introduction to the GTL mechanism following  Gelfreich and Turaev \cite{2}}\label{section: intro2GT}
In this section, we give a brief introduction to the GTL mechanism following  Gelfreich and Turaev \cite{2}. Consider a Hamiltonian system $H = H(p, q, \eps t),\ (p,q)\in \R^{2n}$ with $\eps$ small. It is routine to consider the frozen system in adiabatic invariant theory
$H = H(p, q, \nu)$, where $\nu=\eps t$ is treated as a parameter.
It is required that the frozen system has a chaotic behavior, namely there exists uniformly-hyperbolic, compact, transitive,
invariant set $\Lambda_{h\nu}$ in every energy interval $H=h\in [h_-,h_+]$ for all $\nu$. In every given
energy level, the set $\Lambda_{h\nu}$ is in the closure of a set of hyperbolic periodic orbits
each of which has an orbit of a transverse heteroclinic connection to any of the
others. This means that orbits of the frozen system may stay close to any of the periodic orbits
for an arbitrary number of periods, then come close to another periodic orbit
and stay there, and so on. Now we take two periodic families $\gm_1$ and $\gm_2$ of the frozen system. It is shown that under
some natural conditions one can arrange jumps between $\gm_1$ and $\gm_2$ in such a way that the energy keeps growing.
It is proved that
\begin{Thm}[Theorem 2 and 3 of \cite{2}]
Consider \[v_i(h,\nu) = \dfrac{1}{T_i}
\int^{T_i}_0 \dfrac{\partial H(p,q,\nu)}{\partial\nu}\Big|_{(p,q)=\gm_i(t;h,\nu)}dt, \quad i=1,2,\]
where $T_i$ is the period of the periodic orbit $\gm_i$. Assume that the differential equation
\[\dfrac{dh}{d\nu}= \max\{v_1(h, \nu), v_2(h, \nu)\}-\sigma\beta(h,\nu) \]
has a solution $h_\sigma(\nu)$ for $\sigma$ sufficiently small to suppress $\beta$ $($where $\beta$ is defined in equation $(46)$ of \cite{2}$)$.  Assume the uniformity assumptions $[UA1]$ and $[UA2]$ hold true. Then for all
sufficiently small $\eps$ the Hamiltonian system $H(p,q,\eps t)$ has a solution $(p(t),q(t))$ such that \[H(p(0),q(0), 0) = h_\sigma(0),\quad H(p(t),q(t),\eps t) = h_\sigma(\eps t).\]\label{Thm: GT}
\end{Thm}
We shall show $h_\sigma$ grows linearly. Note that the uniformity assumptions $[UA1],[UA2]$ are automatically fulfilled for any compact set of $h$ and $\nu$, which is exactly what we consider. So we do not cite the lengthy formulation of $[UA1],[UA2]$.\\

\section{The configuration of the four-body problem}\label{section: per}
In this section, we first establish the quasi-periodic motion of the Sun-Jupiter-Planet system, then write down the Hamiltonian governing the motion of the asteroid.

Before the proof, let us introduce some notations in the following definition. Different coordinates will get involved for the convenience of proofs, such as the Cartesian coordinates, polar coordinates and Delaunay coordinates. Please go to Appendix~\ref{AppendixDelaunay}, \ref{AppendixPolar} for the derivations and physical meanings of them.
\begin{Def}
\begin{enumerate}
\item In the Cartesian coordinates, we use $(x,\dot{x}, y, \dot{y})\ ($or $(q, p))$, where $(x,y)\ ($or $q)$ is the position and $(\dot{x},\dot{y})\ ($or $p)$ is the velocity.
\item In Delaunay coordinates, we use the variables\textit{ $(L,\ell,G,g)$}, where $L^2$ is the semimajor of the Keplerian ellipse, $\ell$ the mean anomaly, $G$ the angular momentum and $g$ the argument of the perihelion.
\end{enumerate}
\end{Def}
\begin{Def}
\begin{enumerate}
\item We use the subscripts ``S,J,P,A" to indicate the corresponding quantity of a certain body. For example, $G_P$, the angular momentum of the planet. $r_P$, the relative distance from the planet to the origin.
\item The notations $r_{PJ},r_{PS},r_{AJ},r_{AS},r_{AP}$ denote the mutual distances between the planet and the Jupiter, the planet and the sun, the asteroid and the Jupiter, the asteroid and the sun, the asteroid and the planet, respectively.
\item The notation $r_1$ $($respectively $r_2)$ denotes the distance a point at $(x,y)$ to the point $(-\mu,0)$ $($respectively $(1-\mu,0))$ on the $x$-axis.
\end{enumerate}\label{def: 2}
\end{Def}
\subsection{Selection of quasi-periodic orbits of the P3BP formed by Sun-Jupiter-Planet}\label{subsection: SJPper}
We first show the existence of periodic orbits in the S-J-P system when $\dt =0, \mu\neq 0$, and Sun-Jupiter has circular orbit, i.e. the RPC3BP. Next we continue the periodic orbits to the case of $\dt>0$. We want our periodic orbit to have long period of order $\eps^{-1}$. 

It is proven by Arenstorf and Barrar \cite{9} that the periodic orbits that are symmetric along the $x$-axis are locally isolated on the energy level and therefore can be continued to the RPC3BP for $\mu>0$. The item $(a)$ the next lemma is the result in \cite{9} and is enough for our purpose to prove Theorem \ref{Thm: main} for small $\mu$. Sometimes, people perform numerical study of RPC3BP for realistic value mass $\mu\simeq 10^{-3}$, in which case we need item $(b)$. The idea is that if the orbit of the planet is faraway from the two primaries we can treat the gravitational force as coming from the mass center of the primaries so that we get a perturbed Kepler motion. A similar result for large $\mu$ and for nearly circular orbit of the planet is contained in Chapter 9 of \cite{8}. We modify the proof of the \cite{9} to get item $(b)$. Part (c) will be used to check a nondegeneracy condition in Lemma \ref{Lm: continuition}.
\begin{Lm}
Suppose when $\mu=0$, $y_P(0)=0$ and $\dot x_P(0)=0$, i.e. at time $t=0$, the planet crosses the $x$-axis perpendicularly, and for a later time $T/2$, where $T/2\pi\in \Z$, the planet again crosses the $x$-axis perpendicularly, i.e. $y_P(T/2)=0,\ \dot x_P(T/2)=0$. Then

(a)  except at most finitely many eccentricity  $e_P$, there exists a periodic orbit of the RPC3BP(S-J-P) for each $e_P$ and for $\mu>0$ sufficiently small.

(b) If $T$ is large enough and the eccentricity satisfies $0<e_P<1/2$ when $\mu=0$, then for any $0<\mu\leq 1/2$, there exists periodic orbits of period slightly different from $T$.

(c) If $T$ is large enough and the eccentricity satisfies $0<e_P<1/2$ when $\mu=0$, then  except at most finitely many eccentricity  $e_P$, we have that the period $T_\mu$ of the continued periodic orbit in the RPC3BP with $\mu>0$ differs from the value $T$ and
\[0< |T_\mu-T|\ll 2\pi.\]
\label{Lm: barrar}
\end{Lm}

The next lemma enables us to continue the periodic orbits found in Lemma \ref{Lm: barrar} to the S-J-P system with $\dt>0.$
\begin{Lm}\label{Lm: continuition}
 $($Theorem 9.6.1 in \cite{8}$)$\\
Any elementary periodic solution of the planar restricted three-body problem whose period $T$ is not an integer multiple of $2\pi$ can be continued into the full three-body problem with one small mass. The continued orbit is quasi-periodic in the eight dimensional phase space in Jacobi coordinates $($see \eqref{EqJacobi} later$)$ where angular momentum conservation is not reduced.
\end{Lm}
Here ``elementary" means that the periodic orbit is isolated on the energy surface. For the sake of completeness and notational convenience, we will sketch the proof.

Now we show the existence of quasi-periodic orbits for the P3BP(S-J-P) for sufficiently small $\dt,\eps$.
We first apply Lemma \ref{Lm: barrar} to find a periodic orbit in the RPC3BP (S-J-P) with $\dt=0,\mu>0$. Our periodic orbit is isolated on its energy level since it is obtained from symmetric orbits along the $x$-axis in the case of $\mu=\dt=0$, which are isolated. Item (c) of Lemma \ref{Lm: barrar} shows that $T_\mu/(2\pi)$ is not an integer if $T/(2\pi)$ is. Next, we apply Lemma \ref{Lm: continuition} to get that the periodic orbit that we obtained using Lemma \ref{Lm: barrar} can be continued to the full three-body problem for sufficiently small $\dt$ and we get a quasi-periodic orbit. The resulting quasiperiodic orbit has the following two properties:
\begin{itemize}
\item  Orbits of the sun and Jupiter are nearly circular.
\item  Orbit of the planet is nearly elliptic with eccentricity $0<e<\dfrac{1}{2}$ and large semimajor $O(\eps^{-2/3})$.
\end{itemize}We will quantify the two properties later in Lemma \ref{LmEst}.

Next we give the proof of Lemma \ref{Lm: barrar}.
\begin{proof}[proof of Lemma \ref{Lm: barrar}]
Part $(a)$ is done in \cite{9}. We recall the argument here and modify it slightly to show part $(b)$.

The idea is as follows. Suppose we have $y(0)=0,\ \dot x(0)=0$, i.e. initially the orbit crosses $x$-axis perpendicularly. If we could show that $y(T/2)=0,\ \dot x(T/2)=0$ for positive $\mu$, then we get a periodic orbit of period $T$.

The Hamiltonian for RPC3BP in Delaunay coordinates $(L_P,\ell_P,G_P,g_P)$ and rotating coordinates is given by (c.f. \cite{7}, \cite{8}):
 \begin{equation}H_P(L_P,\ell_P,G_P,g_P)=-\dfrac{1}{2L^{2}_P}-G_P+\Delta H, \  \mathrm{where}\  \Delta H=\dfrac{1}{r_P}-\dfrac{\mu}{r_{PJ}}-\dfrac{1-\mu}{r_{PS}}.\label{eq: rpc3bp}\end{equation}

Here we write the perturbation using the polar coordinates for simplicity. It should be converted into the Delaunay variables. Rotating coordinates are the noninertia coordinates we choose to fix the sun and Jupiter on the $x$-axis. In the following proof, we suppress the subscript $P$ for Delaunay variables but maintaining it for $r_P,r_{PS},r_{PJ}$.

The method is to consider the double Kepler problem (2BP (S-P) + 2BP (S-J)) by setting $\mu=0$, then the Hamiltonian in rotating coordinates becomes \begin{equation}H(L,\ell,G,g)=-\dfrac{1}{2L^{2}}-G,\label{eq: 2kepler}\end{equation}
and the Hamiltonian equations are
\[\dot{ L}=\dot{ G}=0,\quad \dot{ \ell}=\dfrac{1}{ L^3},\quad \dot{ g}=-1.\]

Consider the resonance relation $\dot{ g}=-(m/k)\dot{ \ell}$, where $m,k\in\Z$  are relative primes. We choose the number $k$ small but $m$ large and define $\eps=k/m$.
This gives us $L=\eps^{-1/3}$ and the semimajor $a=L^2=\eps^{-2/3}.$ For eccentricity $e=\sqrt{1-(b/a)^2}\leq 1/2$ where $b$ is the semi-minor, we get \[r_P\geq a-\sqrt{a^2-b^2}=a(1-e)\geq \dfrac{1}{2}\eps^{-2/3}.\]
Moreover, since we have $e=\sqrt{1-(G/L)^2}$ in the Appendix \ref{AppendixDelaunay}, the assumption $e\leq 1/2$ also implies $G\geq \dfrac{\sqrt{3}}{2}L=\dfrac{\sqrt{3}}{2}\eps^{-1/3}.$

The period for the periodic orbit is $T_0=2\pi m=2\pi k/\eps$. Our assumption on $y(0),\ \dot x(0), \ y(T/2),\ \dot x(T/2)$ can be reformulated in terms of Delaunay coordinates:\\
 {\it initially, we have
$g(0)=-\pi,\ell(0)=\pi,\ $
and at the half period, we have
\begin{equation}\label{eq: barrar}g(T/2)=-(1+m)\pi,\quad \ell(T/2)=(1+k)\pi.\end{equation}}
For $\mu=0$, it follows from the Hamiltonian equations that
\[L=\eps^{-1/3},\quad G=\mathrm{const},\quad \ell=t/L^3+\pi,\quad g=-t-\pi.\]

We perform Taylor expansion of $\Delta H$ for large $r_P$ to get
{\small\begin{equation}\label{EqHeqRPC}
\begin{aligned}
&\Delta H=-\dfrac{1}{2r_P^3}(\mu|r_J|^2+(1-\mu)|r_S|^2)+\dfrac{3}{2r_P^5}(\mu(q_J\cdot q_P)^2+(1-\mu)(q_S\cdot q_P)^2)+O(1/r_P^4) \\
&=-\dfrac{\mu(1-\mu)}{2r^3_P}\left(1-\dfrac{3}{r_P^2}\langle (1,0),q\rangle^2\right)+O(1/r_P^4)\\
& \dot L=-\dfrac{\partial \Delta H}{\partial \ell}=-\dfrac{\partial \Delta H}{\partial q}\dfrac{\partial q}{\partial \ell}=O(\mu\eps^2),\quad \dot \ell=\dfrac{1}{L^3}+\dfrac{\partial \Delta H}{\partial L}=\dfrac{1}{L^3}+\dfrac{\partial \Delta H}{\partial q}\dfrac{\partial q}{\partial L}=\dfrac{1}{L^3}+O(\mu\eps^{7/3}),\\
& \dot G=-\dfrac{\partial \Delta H}{\partial g}=-\dfrac{\partial \Delta H}{\partial q}\dfrac{\partial q}{\partial g}=O(\mu\eps^2),\quad \dot g=-1+\dfrac{\partial \Delta H}{\partial G}=-1+\dfrac{\partial \Delta H}{\partial q}\dfrac{\partial q}{\partial G}=-1+O(\mu\eps^{2}),\\
\end{aligned}
\end{equation}}
where we used \[r_P\geq (1/2)\eps^{-2/3},\quad  \dfrac{\partial \Delta H}{\partial q}=O(\mu/r_P^4)=O(\mu \eps^{8/3}),\quad \mathrm{as}\ \eps\to 0,\] as well as
the computations \[\left(\dfrac{\partial }{\partial L},\dfrac{\partial }{\partial \ell},\dfrac{\partial }{\partial G},\dfrac{\partial }{\partial g}\right)q=O(\eps^{-1/3}, \eps^{-2/3}, \eps^{-2/3},\eps^{-2/3})\] obtained using \eqref{Equl}, \eqref{eq: delaunay}.

For $\mu>0$, after integrating the $\dot g,\ \dot \ell$ equations over time $T/2\simeq \pi k\eps^{-1}$, we get that \eqref{eq: barrar} becomes
\[-\dfrac{T}{2}+O(\mu\eps)=-m\pi,\quad \dfrac{T/2}{L(0)^3}+O(\mu\eps^{4/3})=k\pi.\]
A little simplification using $m/k=1/\eps$ gives
\begin{equation}\label{EqTL}T+O(\mu\eps)=2\pi m,\quad L(0)+O(\mu\eps)=\left(\dfrac{m}{k}\right)^{1/3},\end{equation}
where the $O$ notation is used either as $\mu\to 0$ or $\eps\to 0.$
For $\mu=0$, solution exists even in the case $\eps$ is not small. So we apply implicit function theorem to get solution $T,L$ for small enough $\mu$. This is  how the authors prove part (a) of the lemma in \cite{9}. (Compared with \cite{9}, we use an extra assumption $e\leq 1/2$ to get lower bound for $r_P$. This is not needed for $\mu$ small. Instead we need to exclude finite possible $e$ value to avoid collision. )

Next, we consider part $(b)$. We want to treat $\eps$ to be a small quantity. Since the variables  $T,\ L$ depend on $\eps$, we get rid of the $\eps$ dependence by setting $T=\eps^{-1} \hat T$ and $L=\eps^{-1/3} \hat L$ to get equations for $\hat T,\hat L$
\[\hat T+O(\mu\eps^2)=\eps 2 m\pi=2k\pi,\quad \hat L+O(\mu\eps^{4/3})=\eps^{1/3}\left(\dfrac{m}{k}\right)^{1/3}=1.\]
For $\eps=0$, solution $\hat L,\hat T$ exists. Implicit function theorem implies that for small enough $\eps$, we have solutions $\hat T,\hat L$ regardless of the size of $\mu$, hence we get $T,L$. This proves part $(b)$.

For part $(c)$, we consider the $\dot g$ equation.  The idea is to show the $O(\mu\eps^2)$ perturbation in the $\dot g$ equation gives nonzero contribution after integrating over one period of the ellipse so that in \eqref{EqTL} $T$ cannot be $2m\pi$.
We consider the leading term in $\dfrac{\partial \Delta H}{\partial q}\dfrac{\partial q}{\partial G}$. The remaining terms would be much smaller than the leading term for $\eps$ small (hence $r_P$ large). We need to study
$\dfrac{d}{dG}\left[\dfrac{-1}{r^3_P}\left(1-\dfrac{3}{r_P^2}\langle (1,0),q\rangle^2\right)\right]$.
We consider first the term $\dfrac{d}{dG}\dfrac{-1}{r^3_P}=\dfrac{3}{2|q|^5}\dfrac{\partial |q|^2}{\partial G}$. We have the following computation using \eqref{Equl} and \eqref{eq: delaunay} in Appendix \ref{AppendixDelaunay}
\begin{equation}
\begin{aligned}
&|q|^2=L^4(\cos u-e)^2+L^2G^2\sin^2 u=L^4(1-e\cos u)^2\\
&\dfrac{\partial |q|^2}{\partial G}=2L^4(1-e\cos u)\left(-\dfrac{\partial e}{\partial G}\cos u+e\sin^2 u\dfrac{\partial e}{\partial G}\dfrac{1}{1-e\cos u}\right)=2L^4(e-\cos u)\dfrac{\partial e}{\partial G}.
\end{aligned}\nonumber
\end{equation}
Since we need to integrate along the orbit over time $T$, as an approximation we integrate along the unperturbed
elliptic orbit over one period and estimate the error later on. 
We have \[dt=L^3 d\ell=L^3(1-e\cos u)du\] using \eqref{Equl} since $e$ is constant, so it is enough to integrate $u$ over $2\pi$.
To simplify the calculation, we consider small $e$ but we need to keep in mind that $\dfrac{\partial e}{\partial G}=-G/(L^2e)$ is singular when $e=0$. We get
\begin{equation}
\begin{aligned}
&\dfrac{\partial |q|^2}{\partial G}L^3(1-e\cos u)du=2L^7(e-\cos u)(1-e\cos u)\dfrac{\partial e}{\partial G}du.\\
&\int_0^T\dfrac{3}{2|q|^5}\dfrac{\partial |q|^2}{\partial G} dt =\int_{0}^{2\pi}\dfrac{3}{2|q|^5}\dfrac{\partial |q|^2}{\partial G}L^3(1-e\cos u)du\\
&=\dfrac{3}{2}L^{-3}\dfrac{\partial e}{\partial G}\int_0^{2\pi}(e-\cos u)(1-e\cos u)^{-4}du=-\dfrac{3\pi}{2}\dfrac{G}{L^5}(1+O(e)),\ \quad e\to 0.
\end{aligned}\nonumber
\end{equation}
Notice the dependence on $e$ is analytic in the final integral, so there cannot be an interval of $e$ that vanishes the integral.
Next we consider
\begin{equation*}\begin{aligned}
&\langle (1,0),q_P\rangle^2=(L^2(\cos u-e)\cos g-LG\sin u\sin g)^2\\
&=L^4(\cos u-e)^2\cos^2 g+L^2G^2\sin^2 u\sin^2 g+L^3G\sin u(\cos u-e)\sin 2g,\end{aligned}\end{equation*}
using \eqref{eq: delaunay} rotated by angle $g$. 
Notice we have $\dot g\simeq -1$. When we convert $dt$ to $du$ as above, we get $g$ is fast rotating, $\dot g\simeq -L^3 (1-e\cos u)\dot u=-\eps^{-1}(1-e\cos u)\dot u$. So we can replace $\cos^2 g,\ \sin^2 g$ by $1/2+o(1)$ and $\sin 2g$ by $o(1)$ when doing the integration w.r.t. $u$ according to Riemann-Lebesgue Lemma, so $\langle (1,0),q_P\rangle^2$ becomes $\dfrac{1}{2}(L^4(\cos u-e)^2+L^2G^2\sin^2 u)(1+o(1))=\dfrac{1}{2}|q|^2(1+o(1))$. Now we can handle the remaining term \begin{equation*}\begin{aligned}
&\int_0^T\dfrac{d}{dG}\left[\dfrac{3}{r_P^5}\langle (1,0),q\rangle^2\right]dt=\dfrac{d}{dG}\int_0^T\left[\dfrac{3}{r_P^5}\langle (1,0),q\rangle^2\right]dt\\
&=\dfrac{d}{dG}\int_0^{2\pi}\left[\dfrac{3}{r_P^5}\langle (1,0),q\rangle^2\right]L^3(1-e\cos u)du\\
&=\dfrac{d}{dG}\int_0^{2\pi}\left[\dfrac{3}{2r_P^3}(1+o(1))L^3(1-e\cos u)\right]du,\quad \eps\to 0,
\end{aligned}\end{equation*}
which goes back to the $\dfrac{d}{dG}\dfrac{-1}{r^3_P}$ case done in the above. Moreover, we see that the two cases do not cancel each other. 

We then find some $0<e<1/2$ not close to 0 such that the final integral is nonvanishing and study the error coming from the fact that we use unperturbed elliptic orbit to approximate the real orbit. As $\eps\to 0,$ the oscillation of $L,G$ within one period is $O(\mu\eps)$ according to \eqref{EqHeqRPC}. In the expression for $\dfrac{\partial |q|^2}{\partial G}$ the leading error term coming from the oscillation of $L,G$ is $O(L^3\mu\eps)=O(\mu)$, so the leading error in $\dfrac{1}{|q|^5}\dfrac{\partial |q|^2}{\partial G}$ is $O(\eps^{10/3}\mu)$. After integrating over time $T\simeq 2k\pi/\eps$ the error is $O(\eps^{7/3}\mu)\ll O(\eps^{4/3})=G/L^5.$ So we get part $(c)$ of the lemma.
\end{proof}
Next, we give the proof of Lemma \ref{Lm: continuition}. The proof is given in \cite{8}, Chapter 9. We follow that proof with some necessary modifications to get properties of the quasi-periodic orbit that will be used later.
\begin{proof}[proof of Lemma \ref{Lm: continuition}]
We write the three-body problem
\[H_3=\left[\dfrac{|p_P|^2}{2\delta}-\sum_{i=J,S}\dfrac{\delta m_i}{|q_i-q_P|}\right]+\left[\sum_{i=J,S}
\dfrac{|p_i|^2}{2m_i}-\dfrac{(1-\mu)\mu}{|q_S-q_J|}\right],\]
where we have $m_J=\mu,\ m_S=1-\mu$. We choose the mass center as the origin so that we have $(1-\mu) q_S+\mu q_J+\dt q_P=0$ and $p_S+p_J+p_P=0$.
We next introduce the following Jacobi coordinates
\begin{equation}\begin{cases}
&q_1=q_J-q_S\\
&q_2=q_P-((1-\mu)q_S+\mu q_J)
\end{cases}\quad \begin{cases}
&p_1=(1-\mu)p_J-\mu p_S\\
&p_2=p_P
\end{cases}\label{EqJacobi}\end{equation}
to reduce the system into the following form (see Chapter 7 of \cite{8})
\[H_3=\left[\dfrac{|p_2|^2}{2\beta}-\dfrac{(1-\mu)\dt}{|q_2+\mu q_1|}-\dfrac{\mu\dt}{|q_2-(1-\mu) q_1|}\right]+\left[\dfrac{|p_1|^2}{2\al}-\dfrac{\al}{|q_1|}\right].\]
where $\al=\mu(1-\mu),\ \beta=\dt/(1+\dt).$ We can also check that the total angular momentum becomes $q_1\times p_1+q_2\times p_2$. We write the system in rotating coordinates to get
\begin{equation}\label{EqH3rot0}H_{3,rot}=\left[\dfrac{|p_2|^2}{2\beta}-\dfrac{(1-\mu)\dt}{|q_2+\mu q_1|}-\dfrac{\mu\dt}{|q_2-(1-\mu) q_1|}-q_2\times v_2\right]+\left[\dfrac{|p_1|^2}{2\al}-\dfrac{\al}{|q_1|}-q_1\times p_1\right].\end{equation}
The second parenthesis is a two-body problem in rotating coordinates, we transform it to polar coordinates $(p_1,q_1)\to (R,r,\Theta,\theta)$ according to the Appendix \ref{AppendixPolar} to get
\[H_{2,rot}=\dfrac{1}{2\al}(R^2+\dfrac{\Theta^2}{r^2})-\Theta-\dfrac{\al}{r}.\]
We also know that the Kepler circular motion is a critical point of the two-body problem in rotating coordinates (see the Appendix \ref{AppendixPolar} for the derivation). So we linearize $H_{2,rot}$ around the circular motion $R=0,\ r=1,\ \Theta=\al,\ \theta=0.$ We use the rescaling \begin{equation}\label{EqTilde}p_2=\dt v_2,\ R=\sqrt{\dt} \tilde R,\ r=1+\sqrt{\dt}\tilde r.\end{equation} We consider only the total angular momentum being $\al+c\dt$ for some constant $c$ so that $\Theta=\al+c\dt-\dt q_2\times v_2$ as the angular momentum $q_1\times p_1$. Moreover, since the linearization is done in a neighbourhood of the Kepler circular motion, we have $q_1=(1,0)+O(\sqrt{\dt})$.
After the linearization, the Hamiltonian system can be written as
\begin{equation}\label{EqH3rot}\dfrac{H_{3,rot}}{\dt^2}=\left[\dfrac{|v_2|^2}{2}-\dfrac{(1-\mu)}{|q_2+\mu (1,0)|}-\dfrac{\mu}{|q_2-(1-\mu) (1,0)|}-q_2\times v_2\right]+\dfrac{1}{2}\left[ \dfrac{\tilde R^2}{\al}+\al \tilde r^2\right]+O(\sqrt{\dt}).\end{equation}

Thus to first order, the Hamiltonian of the full 3-body problem decouples into the sum of the
Hamiltonian for the RPC3BP and a harmonic oscillator.

The phase space is now three degrees of freedom with coordinates $(q_2,v_2,\tilde R,\tilde r)\in \R^6$. We apply the Lyapunov center theorem. When $\dt=0$, we have a periodic orbit which is the product of the fixed point $\tilde R=0,\ \tilde r=0$ and the periodic orbit of the RPC3BP constructed in Lemma \ref{Lm: barrar}. The period is not an integer multiple of $2\pi$ according to part (c) of Lemma \ref{Lm: barrar}, so that we can apply Lyapunov center theorem to get a periodic orbit of the three body problem in coordinates $(q_2,v_2,\tilde R,\tilde r,)$ for $\dt>0$ (see Chapter 9.6 of \cite{8} for more details). If we view the resulting periodic orbit in the eight dimensional space with coordinates $(q_2,v_2,\tilde R,\tilde r,\Theta,\theta)$, we get one more angular variable $\theta$ so that we get a quasi-periodic orbit.
\end{proof}
\begin{Lm} \label{LmEst} In the S-J-P three-body problem, if we assume $\dt=O(\eps^{3})$ as $\eps\to 0$ and we also assume the total angular momentum is $\al+c\dt$ for some constant $c$. Then for the orbit in the previous Lemma \ref{Lm: continuition}, we have \\
(a) the following estimates for the planet
\begin{equation}
\begin{aligned}
&|L_P(t)-L_P(0)|,\quad |G_P(t)-G_P(0)|\leq C \eps+o(1)\quad \mathrm{as}\ \dt\to 0,\\
\end{aligned}\end{equation}
(b) the following estimates for the motion of the sun and Jupiter
\[(\tilde R,\tilde r)=o(\sqrt{\dt}),\quad \dt\to 0,\]
where $\tilde R,\ \tilde r$ are defined in \eqref{EqTilde} and $R,\ r$ therein are among the polar coordinates of $p_1,q_1$ in \eqref{EqJacobi} using Appendix \ref{AppendixPolar}.

(c) and the following estimates for Jupiter
\[\dot{\theta}=-2\sqrt{\dt}\tilde r+\dfrac{1}{\al}(c\dt -q_2\times p_2)+O(\dt^{3/2}),\quad \dt\to 0\]
where $c$ is a constant, $\sqrt{\dt} \tilde{r}=r-1$ is the radial deviation from the circular motion and $\theta$ is the polar angle in the rotating coordinates of uniform angular velocity 1.
\end{Lm}
\begin{proof}
To show part (a), notice the motion of the planet is periodic with period $O(1/\eps)$. When $\dt=0$, we integrate the $\dot L,\ \dot G$ equations in \eqref{EqHeqRPC} over one period $O(1/\eps)$ to get that the oscillation of $L,G$ during one period is $O(\eps) $ as $\eps\to 0$. When $\dt>0$ but small, the difference of the value $L(t),G(t)$ from that of the case $\dt=0$ is $o(1)$. Namely we have {\small\[|L_{\dt>0}(t)-L_{\dt>0}(0)|\leq |L_{\dt>0}(t)-L_{\dt=0}(t)|+|L_{\dt=0}(t)-L_{\dt>0}(0)|+|L_{\dt=0}(0)-L_{\dt>0}(0)|\leq C\eps+o(1).\]}
Next, to get part (c) we linearize the Hamiltonian $H_{2,rot}$ around $r=1,\ \Theta=\al$. We use $\sqrt{\dt}\tilde r=r-1$ and $\Theta=\al+c\dt-q_2\times p_2$.

Finally, we consider part (b). We formally write the linearized Hamiltonian equation for two-body part in \eqref{EqH3rot} as $\dot Y=AY$, which is a harmonic oscillator of period $2\pi$. Next, we write the equation for $Z:=(\tilde R,\tilde r)$ as $\dot Z=AZ+O(\sqrt{\dt}|Z|^2+\dt\eps^{-1/3}+\sqrt{\dt} \eps^{4/3})$. To see the form of the perturbation, we analyze the $O(\sqrt{\dt})$ term in \eqref{EqH3rot}. The $O(\sqrt{\dt}|Z|^2)$ term comes from the cubic and higher order terms in the Hamiltonian $H_{2,rot}$ after linearization keeping $\Theta$ constant. The $O(\sqrt{\dt}\eps^{4/3})$ comes from linearizing $q_1$ around $(1,0)$ in the RPC3BP part of \eqref{EqH3rot0} noticing $|q_P|\geq c\eps^{-2/3}$. Finally, $O(\dt\eps^{-1/3})$ comes from the oscillation of $\Theta$, which is $c\dt+q_2\times p_2=O(\dt\eps^{-1/3})$ since $q_2\times v_2=G$ is the angular momentum of the RPC3BP proportional to $L=\eps^{-1/3}$ due to the fact $e=\sqrt{1-(G/L)^2}<1/2$.  We then take the difference to get
\[(Z-Y)'=A(Z-Y)+O(\sqrt{\dt}|Z|^2+\dt\eps^{-1/3}+\sqrt{\dt} \eps^{4/3}).\]
The solution has the form \[Z(t)=Y(t)+e^{At}(Z-Y)(0)+\int_0^t e^{A(t-s)}O(\sqrt{\dt} |Z|^2+\dt\eps^{-1/3}+\sqrt{\dt} \eps^{4/3})(s)ds.\]
We compare $Z$ with a solution of the harmonic oscillator solution $Y$ with the same initial condition so that
\[Z(t)=Y(t)+\int_0^t e^{A(t-s)}O(\sqrt{\dt} |Z|^2+\dt\eps^{-1/3}+\sqrt{\dt} \eps^{4/3})(s)ds.\]
The harmonic oscillator has period $2\pi$ and the periodic orbit that we constructed in Lemma \ref{Lm: continuition} has period $O(1/\eps)$. Moreover, we have $|e^{At}|$ is uniformly bounded since $\dot Y=AY$ is a harmonic oscillator. When $t$ is less than one period $O(1/\eps)$, we have for $\dt=O(\eps^{3})$ as $\eps\to 0$ that
\[Z(t)=Y(t)+O(\sqrt{\dt}/\eps |Z(s)|^2+\dt\eps^{-4/3}+\sqrt{\dt}\eps^{1/3})=Y(t)+O(\sqrt{\dt}/\eps |Z(s)|^2)+o(\sqrt{\dt}).\]
We know $Y(t)=o(1)$ as $\dt\to 0$ since $Z(t)=(\tilde R,\tilde r)\to 0$ as $\dt\to 0$. Now assume $Y(t)$ is not $o(\sqrt{\dt})$, which means $\lim\sup_\dt\max_t|Y(t)/\sqrt{\dt}|>0.$ We denote by $Z/\sqrt{\dt}=\bar Z$ and $Y/\sqrt{\dt}=\bar Y$. So we get $\bar Z=\bar Y+o(1)$ as $\dt\to 0$ and $\eps$ fixed. We know $\bar Z$ is periodic whose period is close to that of the periodic solution of RPC3BP. i.e. $T=O(1/\eps)$ in part (c) of Lemma \ref{Lm: barrar}, esp. $T/(2\pi)\notin \Z$. However $\bar Y$ is harmonic oscillator of period $2\pi.$ So we get in the LHS of $\bar Z=\bar Y+o(1)$ that $\bar Z(0)=\bar Z(T)$, which forces the RHS to satisfy $\bar Y(0)=\bar Y(T)+o(1)=\bar Y(T\ \mathrm{mod}\ 2\pi)+o(1)$. This cannot be true for $\dt$ small enough since $\bar Y(t)$ is just a nonvanishing multiple of $(\cos t,\sin t)$ up to a time translation in the limit $\dt\to 0$.
\end{proof}

\subsection{ The motion of the asteroid }\label{subsection: asteroid} In this section, we consider the motion of the asteroid under the gravitational force from the sun, Jupiter, and the planet.
The following theorem establishes that the Hamiltonian describing the motion of the asteroid in the 4BP(S-J-P-A) is quasi-periodic perturbation of the RCP3BP(S-J-A).  This is the Hamiltonian for which we prove the existence of diffusing orbits.
\begin{Thm}\label{Thm: ham}
If the motion of the full three-body problem S-J-P is chosen to be a quasi-periodic motion constructed in Lemma \ref{Lm: continuition}, then for sufficiently small $\eps,\dt$ and $\dt=O(\eps^{3})$, the Hamiltonian of the motion of the asteroid in the rotating coordinates can be written into the form of a RPC3BP with a small and slow quasi-periodic perturbation:
\begin{equation}
H_{A,Rot}(L_A,\ell_A,G_A,g_A,\eps t)=-\dfrac{1}{2L^{2}_{A}}-G_{A}+\Delta H+f(\ell_{A},L_{A},g_{A},G_{A},\eps t),\label{eq: main}
\end{equation}
where $(L_A,\ell_A,G_A,g_A)\in T^*\T^2$. The perturbation $f$ has the form
\begin{equation}
\begin{aligned}
&f(\ell_{A},L_{A},g_{A},G_{A},\eps t):=-\dot{\theta}q_A \times p_A +(1-\mu)\left(\dfrac{1}{r_1}-\dfrac{1}{r_{AS}}\right) +\left(\dfrac{\mu}{r_2}-\dfrac{\mu}{r_{AJ}}\right)-\dfrac{\dt}{r_{AP}}.
\end{aligned}\label{eq: pert}
\end{equation}
\end{Thm}
where $\theta$ is defined in part (c) of Lemma \ref{LmEst}. 
Moreover, if we set $\nu=\eps t$, $\nu:=(\nu_1,\nu_2)\in \T^2$, the function $f(\bullet_A,\nu)$ is $o(1)$ as $\dt\to 0$ and $\nu_1$ is $2\pi$ periodic, $\nu_2$ has $o(1)$ frequency in the limit $\dt\to 0$.
\begin{proof}
The Hamiltonian of the motion of the asteroid can be written in the complete form:
\[H_{A}(x_A,y_A,\dot x_A,\dot y_A,\eps t)=\frac{1}{2}\dot{x}_A^{2}+\frac{1}{2}\dot{y}_A^{2}-\dfrac{1-\mu}{r_{AS}}-\dfrac{\mu}{r_{AJ}}-\dfrac{\dt}{r_{AP}},\quad (x_A,y_A,\dot x_A,\dot y_A)\in \R^4.\]
The first two terms are the kinetic energy and the last three terms are the potential energy from the sun, Jupiter and the planet respectively. We set the origin as the mass center of the sun, Jupiter and planet, so that we can use the Jacobi coordinates \eqref{EqJacobi} instead of using $q_S,q_J,q_P$ to describe the background motion. The background motion in Jacobi coordinates is quasi-periodic as we show in the proof of Lemma \ref{Lm: continuition}.

 We write the Hamiltonian in rotating coordinates.  We want that in the rotating coordinates sun and Jupiter lie on a line parallel to the x-axis. As a result this rotating coordinates is not uniform. We use the following transformation.
\[\begin{cases}\ q_A:=\exp(\theta_{1}K)\ (x_A,y_A)^T\\ \ p_A:=\exp(\theta_{1}K)\ (\dot{x}_A,\dot y_A)^T\end{cases},\quad K=\left[\begin{array}{cc}0&-1\\
1&0\end{array}\right]\]
where $\theta_1$ is the polar angle of $q_1=q_J-q_S$ in nonrotating coordinates.

We know that $\exp(\theta_{1}K) $ is an orthogonal matrix, so it is easy to check the following identity.
  \[\left[\begin{array}{cccc}
e^{\theta_{1}K} &   0      \\
0 &   e^{\theta_{1}K}
\end{array}\right]
\left[\begin{array}{cccc}
0 &   \mathrm{Id}      \\
-\mathrm{Id} &   0
\end{array}\right]
\left[\begin{array}{cccc}
e^{\theta_{1}K} &   0      \\
0 &   e^{\theta_{1}K}
\end{array}\right]^{T}
=\left[\begin{array}{cccc}
0 &   \mathrm{Id}      \\
-\mathrm{Id} &   0
\end{array}\right].
  \]
This shows that the change of coordinates is symplectic. The coordinates system $(q_A,p_A)$ is also Cartesian.

In general we have $\dot{\theta}_1\neq$ constant, so the rotating angular velocity of the coordinate frame is not uniform. Instead of $-q_A\times p_A$, the term \[-\dot{\theta_{1}}q_A \times p_A=-\left(1+\dot{\theta}\right)q_A\times p_A\]would appear in the new Hamiltonian as the Coriolis term (c.f. the 6th chapter of \cite{8}), where $\dot{\theta}$ is estimated in part (c) of Lemma \ref{LmEst}. We get the Hamiltonian:
\[H_{A,Rot}=\frac{1}{2}p_A ^{2}-(1+\dot{\theta})q_A \times p_A -\dfrac{1-\mu}{r_{AS}}-\dfrac{\mu}{r_{AJ}}-\dfrac{\dt}{r_{AP}}\]
Now plug in the terms:
\[-\dfrac{1-\mu}{r_1}-\dfrac{\mu}{r_2}+\dfrac{1-\mu}{r_1}+\dfrac{\mu}{r_2}.\]
(Recall $r_1$ and $r_2$ are defined to be the distance from the asteroid to $(-\mu,0),\ (1-\mu,0)$ respectively in Definition \ref{def: 2}).

So the Hamiltonian becomes:
\begin{equation}
\begin{aligned}
& H_{A,Rot} =\left[\frac{1}{2}p_A ^{2}-q_A \times p_A -\dfrac{1-\mu}{r_1}-\dfrac{\mu}{r_2}\right]+\left[-\dot{\theta} q_A \times p_A \right.\\
& \left.\qquad +(1-\mu)\left(\dfrac{1}{r_1}-\dfrac{1}{r_{AS}}\right)+\mu\left(\dfrac{1}{r_2}-\dfrac{1}{r_{AJ}}\right)-\dfrac{\dt^2}{r_{AP}}\right].\\
\end{aligned}\nonumber
\end{equation}
In this expression, the first bracket is the RPC3BP in a uniform rotating coordinates. We estimate the second bracket as $o(1)$ as $\dt\to 0$. We will estimate the second bracket in details in the later Lemma \ref{Lm: nondeg}.

  Since we assume $\dt\ll \eps$, the frequency of the angular variable $\dot\theta=O(\dt)$ in part (c) of Lemma \ref{LmEst} is much smaller than the frequency $\eps$ of the periodic orbit. So we denote by $f(\bullet_A,\eps t)$ the terms in the second bracket above by setting $t=\dfrac{1}{\eps}\cdot \eps t$. The function $f(\bullet_A,\nu)$ is $o(1)$ as $\dt\to 0$. We write $\nu=(\nu_1,\nu_2)\in \T^2$ so that $\nu_2=\theta/\eps$ and $\nu_1$ corresponds to the periodic part $(p_2,q_2,\tilde R,\tilde r)$. Then $\nu_1$ is $2\pi$ periodic and $\nu_2$ has $o(1)$ frequency in the limit $\dt\to 0$. This completes the proof.
\end{proof}

\section{Heteroclinic cycle in the RPC3BP (S-J-A)}\label{section: SJA}
In this section we neglect the planet in the 4BP(S-J-P-A) as defined above in Section ~\ref{section: per}.  The goal is prove the existence of two normally hyperbolic periodic orbits and explain the known numerical evidences on the existence of transverse heteroclinic connections between them for the RCP3BP(S-J-A).
\subsection{The existence of two hyperbolic periodic orbits in the RPC3BP}\label{subsection: hyperbolic}
We describe the periodic motions $\gm_1$ (resp. $\gm_2$) of the asteroid near the Lagrangian point $L_1$ (resp. $L_2$).
Here, it is more convenient for us to study the motion in $(x,y,\dot x,\dot y)$ Cartesian coordinates. In the rotating coordinates, the equations of motion are (c.f.\cite{7}):
\[\begin{cases}
\ddot{x}-2\dot y=\Omega_{x},\\
\ddot{y}+2\dot x=\Omega_y,\end{cases}
\mathrm{where}\ \Omega=\dfrac{x^2+y^2}{2}+\dfrac{1-\mu}{r_1}+\dfrac{\mu}{r_2}+\dfrac{\mu(1-\mu)}{2},\]
and $r_1,r_2$ are defined in Definition \ref{def: 2}.
It is easy to check that these equations have an integral---the Jacobi integral.
\begin{equation}J(x,y,\dot x,\dot y):=-(\dot x ^2+\dot y ^2)+2\Omega=-2h,\label{eq: Jacobi}
\end{equation} where $h$ is the energy of the RPC3BP.
The RPC3BP has 5 equilibria called Lagrangian points (see Figure 2). We want to pay attention to the collinear equilibria $L_1, L_2$ lying on the x-axis. The $L_1$ and $L_2$ Lagrangian points are the two positive critical points of the following expression that is the potential $\Om$ restricted on the $x$-axis (c.f. Chapter 2.5 of \cite{7}):
\begin{equation}\dfrac{x^2}{2}+\dfrac{1-\mu}{|x+\mu|}+\dfrac{\mu}{|x+\mu-1|}.\label{eq: l1l2}\end{equation}
The following lemma is a classical result following directly from \eqref{eq: l1l2}.
\begin{Lm}
In the RPC3BP, the distance from Jupiter to the Lagrangian points $L_1$ and $L_2$ is $\pm 3^{-1/3} \mu^{1/3}+o(\mu^{1/3})$ as $\mu\to 0$. We choose $+$ for $L_2$ and $-$ for $L_1.$\label{Lm: 1/3}
\end{Lm}
\begin{proof}
We write ~(\ref{eq: l1l2}) as
\[\dfrac{x^2}{2}+\dfrac{1-\mu}{x+\mu}\pm \dfrac{\mu}{x+\mu-1},\]
where we choose $+$ for $L_2$ and $-$ for $L_1.$

Taking derivative and setting the derivative to be $0$, we get an equation
\[x-\dfrac{1-\mu}{(x+\mu)^2}\mp\dfrac{\mu}{(x-1+\mu)^2}=0.\]
As a first step approximation, we suppose $x=1-\mu+C \mu^\al$ for some undetermined constants $C$ and $\al$ and plug it into the equation to get that $\al=1/3$ and the coefficient for the leading term, i.e. the $\mu^{1/3}$ term
\[3C\mp C^{-2}=0.\]
Solving this we get the lemma.
\end{proof}
Since $L_1$ and $L_2$ are critical points of the potential~(\ref{eq: l1l2}), we can linearize the Hamiltonian system in a neighborhood of the two points.
\begin{Lm} $($Proposition 2 of \cite{4}$)$
The linearized systems of RPC3BP in a neighborhood of $L_1$ and $L_2$ have eigenvalues $\lambda_{1,2} =\pm \lambda, \lambda_{3,4} =\pm i\kappa$, the corresponding eigenvectors, $u_1, u_2, w_1, w_2$, and the general solution:\\
\begin{equation}
u(t)=\alpha_1 u_1 e^{\lambda t}+\alpha_2 u_2 e^{-\lambda t}+2Re(\beta e^{i\kappa t} w_1).
\label{eq: conley}
\end{equation}\label{Lm: ev}
\end{Lm}
From this lemma we see that for the linearized system the two complex conjugate purely imaginary eigenvalues give rise to periodic orbits, and the two real eigenvalues give the stable and unstable directions. Next, an application of the Lyapunov center theorem (c.f. \cite{4} and \cite{8}, Theorem 9.2.1) shows that for energy levels slightly higher than the critical energy levels, i.e. the energy level that we obtained by setting $\dot x,\dot y=0$ and plug in the position of $L_1,L_2$ in \eqref{eq: Jacobi}.
\subsection{Heteroclinic intersections in the RPC3BP}\label{subsection: hetero}
We have shown the existence of two Lyapunov periodic orbits in RPC3BP. However, the existence of transversal intersection of their stable and unstable manifolds remains an open problem. In this section, we cite some relevant results in favor of our transversality assumption in Theorem \ref{Thm: main}.
\subsubsection{The rigorous numerical result for one special energy level and realistic mass ratio $\mu$}
In \cite{5}, the authors proved the following two theorems from \cite{5} using rigorous numerics. One of the motivations of the work \cite{5} was to justify the numerical evidences in \cite{12} showing transversal intersection of stable and unstable manifolds of the Lyapunov orbits.
\begin{Thm}
For RPC3BP with
$J= 3.03, \mu= 0.0009537$, there exist two periodic solutions in the Jupiter region, $\gamma_1$ and $\gamma_2$, called Lyapunov orbits, and
there exists heteroclinic connections between them, in both directions.
\end{Thm}

\begin{Thm}
For RPC3BP with $J= 3.03, \mu= 0.0009537$, there exist a symbolic dynamics on four symbols
${S, X, L_1, L_2}$ corresponding to sun and exterior
regions and vicinity of $L_1$ and $L_2$, respectively.
\end{Thm}
\subsubsection{Numerical evidence for the Hill problem and RPC3BP}
Now we cite the result from \cite{18}.
In \cite{18}, the authors study the Hill problem that is a simplification of the RPC3BP in the limit $\mu\to 0$ taken as follows. We consider the following (c.f.\cite{18}).
\begin{equation}\label{EqHill}x=X\mu^{1/3}+\mu-1,\quad \dot x=\dot X \mu^{1/3},\quad y=Y \mu^{1/3},\quad \dot y=\dot Y \mu^{1/3}. \end{equation}
The rescaling is natural in view of Lemma~\ref{Lm: 1/3}.
Then we express the Jacobi constant $J$ in terms of $(X,Y,\dot X,\dot Y)$, to obtain the Jacobi constant in the Hill's limit
\begin{equation}J_H(X,Y,\dot X,\dot Y)=-\dot X^2-\dot Y^2+\dfrac{2}{(X^2+Y^2)^{1/2}}+3X^2+O(\mu^{1/3}),\label{eq: hill}\end{equation}
where $J_H=\mu^{-2/3}(J-3(1-\mu))$ and $\mu^{2/3} J_H\to J-3$ as $\mu\to 0$.\\

Numerics in \cite{18} shows that

{\it for the Hill problem \eqref{eq: hill}, for $\dfrac{1}{2}|J_H|^{-3/2}> 1/18$ the invariant manifolds of the Lyapunov periodic orbits give rise to homoclinic and heteroclinic orbits, connecting a vicinity of $L_1$ to itself and to one of $L_2$ and vice versa. These results extend to the RPC3BP for $\mu$ sufficiently small because of transversality.}

If this result is established rigorously, our transversality assumption of Theorem \ref{Thm: main} would be satisfied for sufficiently small $\mu$. In $J_H$, we set $\dot X,\ \dot Y,\ \mu=0$ and find the two fixed points $L_1,L_2$ lie on the same energy level. So the diameters of the Lyapunov periodic orbits can be as small as we wish on the same energy level and for small $\mu$.
\section{Existence of heteroclinic cycle for four-body problem}\label{section: hyp}
In this section, we show that under the transversality assumption in Theorem \ref{Thm: main},  our four-body problem has the heteroclinic cycle required by the GTL mechanism. 
\subsection{Translation of Section~\ref{section: SJA} into the language of classical hyperbolicity theory.}\label{subsection: language}
Recall that the RPC3BP has two degrees of freedom and its phase space is $T^*\T^2$.  Denote the flow of the RCP3BP by $\Phi_t$.  Fixing the energy restricts the dynamics to a 3 dimensional energy surface denoted by $M (h), h \in[h_- , h_+ ]$. (In the following sections, we use $h$ instead of the Jacobi constant $J$ to stand for the energy to be consistent with the main result, $J=-2h$.) On every energy surface $M(h)$, there are 2 hyperbolic periodic orbits $\gm_1 (h), \gm_2 (h)$ (see Theorem ~\ref{Thm: con}).
\begin{Def}
Define the cylinder $C_{i}=\bigcup_{h} \gm_{i}(h)$, $i=1,2$. These cylinders are two dimensional, and each is diffeomorphism to $[h_-, h_+]\times \T^1$. Additionally their stable and unstable manifolds, denoted $W^{s} (C_i )$ and $W^{u}(C_i )$ respectively are both three dimensional.
The heteroclinic intersections of these manifolds  \[\Gamma_{12}=\bigcup_h \Gamma_{12}(h)\subset W^{s}(C_{1})\cap W^{u}(C_{2}),\quad \Gamma_{21}=\bigcup_h \Gamma_{21}(h)\subset W^{u}(C_{1})\cap W^{s}(C_{2})\] are two dimensional.
\end{Def}
The following lemma is a translation of Section~\ref{section: SJA} into the language of classical hyperbolicity theory.
\begin{Lm}$($Lemma 3.2 of \cite{1}$)$
For some constant $k$, $\lambda>0$ $($calculated at the end of Section 4.1 $)$, and for all $x\in \gm_i(h)$, we have the decomposition of the tangent space into stable, unstable, and central subspaces.
\[T_{x}M(h)=E_{x}^{s}\oplus E_{x}^{u}\oplus T_{x}\gm_i(h)\]
with \[ \Vert D\Phi_{t}(x)|_{E_{x}^{s}} \Vert\leq k e^{-\lambda t},\ \textrm{for}\ t>0,\]
\[\Vert D\Phi_{t}(x)|_{E_{x}^{u}} \Vert\leq k e^{\lambda t},\ \textrm{for}\ t<0,\]
\[\Vert D\Phi_{t}(x)|_{T_{x}\gm_i(h)} \Vert\leq k,\ \textrm{for}\ t\in \R.\] The stable and unstable manifolds to $C_{i}$: $W^{s(u)}(C_{i})$, are three dimensional manifolds
diffeomorphic to $[h_- , h_+] \times \T^{1} \times \R$, and their heteroclinic intersections $\Gamma_{12}$ and $\Gamma_{21}$ are both diffeomorphic to $[h_{-} , h_{+}] \times \R$.
\end{Lm}
\subsection{Add the perturbation to the RPC3BP $($S-J-A$)$}\label{subsection: perturb}
When the perturbation is added, we start to consider the R4BP whose Hamiltonian is ~(\ref{eq: main}).
The system as stated is non-autonomous. To apply the mechanism of \cite{2}, we consider the frozen system with $\eps t=\nu,\ \nu\in \T^2$.  Note the frozen system is autonomous and its Hamiltonian can be written in the form:
 \begin{equation}H_{A,rot}(\ell_{A},L_{A},g_{A},G_{A},\nu)=-\dfrac{1}{2L^{2}_{A}}-G_{A}+\Delta H+f(\ell_{A},L_{A},g_{A},G_{A},\nu).\label{eq: frozen}\end{equation}
Now we are working with an autonomous system with 2 degrees of freedom. For each fixed $\nu\in \T^2$, using the hyperbolicity theory we have the following lemma. It shows that when $\dt \neq 0$ the cylinders along with their stable and unstable manifolds  and heteroclinic intersections persist.
\begin{Lm}
Consider the frozen system \eqref{eq: frozen} with $\nu\in \T^2$ fixed, $H_{A,rot}\in C^{r}, 2\leq r<\infty$. Then there exists a $\dt_{\nu}$, such that for $|\dt|<\dt_{\nu}$, the perturbed cylinder $C_{i,\dt}$ is hyperbolic, $C^{r-1}$ diffeomorphic to $C_i$, locally invariant and is $\dt$-close to the unperturbed cylinder $C_i$.  Its (un)stable manifolds $W^{s(u)}(C_{\dt,i})$ are also $\dt$-close to the unperturbed $W^{s(u)}(C_i)$ in the $C^{r-2}$ sense. (i=1,2). Moreover, the compactness of $\T^2(\ni \nu)$ gives us a uniform $\dt^*>0$ such that the above statement holds for all $\nu\in\T^2$ and $|\dt|< \dt^*$.\\
\end{Lm}
\begin{proof} Direct application of classical hyperbolicity theory, c.f. Theorem 4.2, and theorem A.14, A.12 of \cite{1}.\end{proof}
\begin{Rk}
Since we have the transversal heteroclinic intersections $\Gamma_{12}\subset W^{s}(C_{1}) \cap W^{u}(C_{2})$ and  $\Gamma_{21} \subset W^{u}(C_{1})\cap W^{s}(C_{2})$, there exist locally unique new heteroclinic intersections:
$ \Gamma_{\dt,12}\subset W^{s}(C_{\dt,1}) \cap W^{u}(C_{\dt,2})$ and  $\Gamma_{\dt,21} \subset W^{u}(C_{\dt,1})\cap W^{s}(C_{\dt,2})$.
$\Gamma_{\dt,12}$ $(\Gamma_{\dt,21})$ is $\dt$-close to $\Gamma_{12}$  $(\Gamma_{21})$ in the $C^{r-2}$ sense, and that $ \Gamma_{\dt,12}$ $(  \Gamma_{\dt,21})$ can be parameterized by a $C^{r-1}$ function on $ \Gamma_{12}$  $(\Gamma_{21})$ to the extended phase space.
\end{Rk}
We still need to show the Lyapunov orbits cannot be broken by the perturbation $f(\nu)$, so that on each energy level of the frozen system, we still have two periodic orbits.
\begin{Lm}
For the frozen system \eqref{eq: frozen} for each $\nu\in \T^2$, each perturbed cylinder is a foliation of periodic orbits.
\end{Lm}
\begin{proof} Because our system is frozen, for fixed $\nu$, it is a system of 2 degrees of freedom. As a matter of fact, the persistence of these perturbed periodic orbits can be established from the Lyapunov center theorem (\cite{8}, Theorem 9.2.1) in the same way as the existence of Lyapunov orbits. The Lyapunov orbits established in Section \ref{section: SJA} is a perturbative result from the linearized system Lemma~\ref{Lm: ev}. If we treat the $f(\nu)$ and the nonlinear part of the RPC3BP expanded at $L_1$ or $L_2$ Lagrangian points as a whole, to perturb the linearized system, we still have a family of Lyapunov periodic orbits due to the same Lyapunov center theorem used by \cite{4} and \cite{8} provided $\dt$ is small. Since we have $\nu\in \T^2$, the compactness of $\T^2$ gives a uniform $\dt$ which works for all $\nu$'s. \end{proof}
\section{Apply the GTL Mechanism to the four-body Problem}\label{section: nondegenerate}
\subsection{Verification of the uniformity assumptions in \cite{2}}\label{subsection: uniformity}
Notice the P4BP under consideration satisfies the GTL mechanism exactly.  Indeed we have the existence of two normally hyperbolic periodic orbits and their transversal heteroclinic intersections on every energy surface of the frozen system in the energy interval $[h_-,h_+]$. In order to apply the result of \cite{2}, we need to check the uniformity assumption [UA1], [UA2] in \cite{2}. However, we do not cite their lengthy formulations here. [UA1] is the hyperbolicity requirement, which is given by Conley's result in Section 4.1, while  [UA2] is trivially satisfied, since we only consider finite energy interval which has compactness.
The remaining thing to do is to check the nondegeneracy condition. 
 \subsection{Verification of nondegeneracy}\label{subsection: nondeg}
We break the proof into several steps. First we write the nondegeneracy condition into a form that we are able to check. Then we show the resulting form depends analytically on the variables. Finally, we use the property of analytic functions to show the nondegeneracy holds.
 \subsubsection{Write integrals responsible for non-degeneracy}
In the Hamiltonian~(\ref{eq: main}), we set $\eps t=\nu$.
According to Theorem~\ref{Thm: GT}, we have \begin{equation}\dfrac{dh}{d\nu}=\max\left\{\dfrac{\pt\bar{f}_1(h,\nu)}{\pt\nu},\dfrac{\pt\bar{f}_2(h,\nu)}{\pt\nu}\right\}-\sigma \bt(h,\nu)\label{eq: 2}\end{equation}
where \begin{equation}\bar{f}_{i}=\dfrac{1}{T_{i}}\int_{0}^{T_{i}}f|_{\gm_i}(l_{A},L_{A},g_{A},G_{A},\nu)dt,\label{eq: integralf}\end{equation}
and $T_{i}\ (i = 1, 2)$ is the period of $\gm_i$. The $\sigma$ here is the $\dt$ in \cite{2}, which is just a small number. Once the $\dt$ in this paper is chosen, we can make $\sigma$ as small as we wish to suppress $\beta$.

Now, integrate both sides.
\begin{equation}h(t_1)-h(0)=\dfrac{1}{2}(\bar{f}_1(t_1)+\bar{f}_2(t_1)-\bar{f}_1(0)-\bar{f}_2(0))+\int^{t_1}_0 \left|\dfrac{d}{d\nu}(\bar{f}_1-\bar{f}_2)\right|d\nu-2\sigma\int \bt d\nu.\label{eq: 3}\end{equation}
We already bound $|\bar f_1|$ and $|\bar f_2|$ by $o(1)$ in Theorem \ref{Thm: ham}. To get linear energy growth, we only need to ensure the following non-degeneracy condition (Theorem 4 in \cite{2}):
 \[\displaystyle \liminf_{t_1\to\infty} \dfrac{1}{t_1}\int ^{t_1}_0 \left|\dfrac{d}{d\nu}(\bar{f}_1(\nu)-\bar{f}_2(\nu))\right|d\nu >0.\]
In our case, the perturbation $f$ is quasi-periodic. So using Birkhoff ergodic theorem, it is sufficient to satisfy:
\begin{equation}\label{eq: f1f2}
\bar{f}_1(\nu)-\bar{f}_2(\nu)\neq \mathrm{const}.\end{equation}
This is also the equation (69) in \cite{2}.

\subsubsection{ The proof of the nondegeneracy condition \eqref{eq: f1f2}.}
\begin{Lm}\label{Lm: nondeg} Consider $\bar f_1$ and $\bar f_2$ defined in~(\ref{eq: integralf}). Assume that $\mu$ sufficiently small and the diameters of the two Lyapunov periodic orbits $\gm_1,\gm_2$ are also sufficiently small.  Then we have for $\dt,\eps$ sufficiently small and satisfying $\dt=O(\eps^{3})$, and the eccentricity of the planet $e_P\in (0,1/2)$  that
\[\bar{f}_1(\nu)-\bar{f}_2(\nu)\neq  \mathrm{const}.\]
\end{Lm}
\begin{proof}
We analyze the expression of $f(\bullet_A,\nu)$ \eqref{eq: pert} term by term.

We denote by $q_{Ai}$ the relative position vector from the asteroid to $L_i,\ i=1,2$. We denote by $q_A=a_i+q_{Ai}$,  where $i=1,2$ means the $L_1,L_2$ Lagrangian points and  $a_i$ is the constant vector from $L_i$ to the origin.

$\bullet$ First we consider first the term using part (c) of Lemma \ref{LmEst}\[-\dot{\theta} q_A ^{T}\times p_A=\left(-2\sqrt{\dt} \tilde r(\nu) +\dfrac{1}{\al}(c\dt +q_2\times p_2)\right)(a_i \times p_A+ q_{Ai} \times p_A)+O(\dt^{3/2}).\] We integrate along the periodic orbits $\gm_i$ holding $\nu$ constant. The integral of $p_A$ is 0 since $p_A$ is a periodic function of zero average, otherwise $\gm_i$ will move away with the average velocity. Next, the term $q_{Ai} \times p_A$ means the angular momentum of the asteroid around the point $L_i$, which we denote by $\omega_i$. So the contribution from this term is $\omega_i \left(-2\sqrt{\dt} \tilde r(\nu) +\dfrac{1}{\al}(c\dt +q_2\times p_2)\right)+O(\dt^{3/2}).$

$\bullet$ Next we consider the term $(1-\mu)\left(\dfrac{1}{r_1}-\dfrac{1}{r_{AS}}\right)$. We denote by $b_i$ the vector from $L_i$ to the $-(\mu,0)$. So $r_1=|q_A+(\mu,0)|=|b_i+q_{Ai}|$ and $r_{AS}=|q_A-q_S|=|q_{Ai}+b_i+\tilde q_{S}|,$ where $\tilde q_{S}=q_S+(\mu,0)$ is the vector from the sun to  $-(\mu,0)$.
We perform Taylor expansion around the points $\tilde q_{S}=0$ to get
\[(1-\mu)\left(\dfrac{1}{r_1}-\dfrac{1}{r_{AS}}\right)=(1-\mu)\dfrac{\langle q_{Ai}+b_i,\tilde q_S\rangle}{|q_{Ai}+b_i|^3}+O(|\tilde q_S|^2),\quad \dt,\eps\to 0.\]

$\bullet$ Next we consider the term $\mu\left(\dfrac{1}{r_1}-\dfrac{1}{r_{AJ}}\right)$. We denote by
$c_i$ the vector from $L_i$ to the $(1-\mu,0)$, $|c_i|=O(\mu^{1/3})$ according to Lemma \ref{Lm: 1/3}. So we have \\ $r_2=|q_A-(1-\mu,0)|=|c_i+q_{Ai}|$ and $r_{AJ}=|q_A-q_J|=|q_{Ai}+c_i+\tilde q_{J}|,$ where $\tilde q_{J}=q_J-(1-\mu,0)$ is the distance from the Jupiter to  $(1-\mu,0)$.
We perform Taylor expansion around the points $\tilde q_{J}=0$ to get
\[\mu\left(\dfrac{1}{r_2}-\dfrac{1}{r_{AJ}}\right)=\mu\dfrac{\langle q_{Ai}+c_i,\tilde q_J\rangle}{|q_{Ai}+c_i|^3}+O(|\tilde q_J|^2),\quad \dt,\eps\to 0.\]

$\bullet$ Finally, we have the term $\dfrac{\dt}{r_{AP}}=\dfrac{\dt}{r_{P}}-\dfrac{\dt\langle q_{A},q_P\rangle}{r_P^3}=\dfrac{\dt}{r_{P}}+O(\dt\eps^{4/3})$ since $r_{AP}\geq C\eps^{-2/3}$.

We have chosen the mass center as the origin, i.e. $\mu \tilde q_J+(1-\mu)\tilde q_S+\dt q_P=0$. We define $\tilde q=\tilde q_J-\tilde q_S$ so that we have \begin{equation}\label{EqTqJS}\tilde q_J=(1-\mu)\tilde q-\dt q_P,\quad \tilde q_S=-\mu \tilde q-\dt q_P.\end{equation}
In our nonuniform rotating coordinates, $q_1=q_J-q_S$ in \eqref{EqJacobi} is parallel to the x-axis.
Moreover, we have converted $(q_1,p_1)$ in \eqref{EqJacobi} to polar coordinates so that we have the following using Lemma \ref{LmEst} part (b)
\[q_1=q_J-q_S=r(1,0)=\tilde q+(1,0),\quad \tilde q=\sqrt{\dt}\tilde r(1,0)=o(\dt).\]
To summarize the above estimates we average along the periodic orbits $\gm_1,\ \gm_2$ and then take the difference to get that
\begin{equation}\label{Eqfinal}\begin{aligned}
&\left(-2\sqrt{\dt} \tilde r(\nu) +\dfrac{1}{\al}(c\dt +q_2\times p_2)\right)(\omega_1-\omega_2)\\
&-(1-\mu) \left[\dfrac{1}{T_1}\int_0^{T_1}\dfrac{ q_{A1}+b_1}{|q_{A1}+b_1|^3}-\dfrac{1}{T_2}\int_0^{T_2}\dfrac{q_{A2}+b_2}{|q_{A2}+b_2|^3} dt\right]\cdot(-\mu \tilde q-\dt q_P)\\
&-\mu\left[\dfrac{1}{T_1}\int_0^{T_1}\dfrac{q_{A1}+c_1}{|q_{A1}+c_1|^3}dt-\dfrac{1}{T_2}\int_0^{T_2} \dfrac{ q_{A2}+c_2}{|q_{A2}+c_2|^3}\, dt\right]\cdot((1-\mu)\tilde q-\dt q_P)+O(|\tilde q_S|^2+|\tilde q_J|^2)
\end{aligned}
\end{equation}
The integrals in the two brackets give us constant vectors which are not zero. We have estimate $\sqrt{\dt}\tilde r,\ \tilde q=o(\dt)$ according to part (b) of Lemma \ref{LmEst}.

Next, we estimate the angular momentum $q_2\times p_2=O(\dt\eps^{-1/3}),$\ $\dt,\eps\to 0$ as we did in the proof of part (b) of Lemma \ref{LmEst} using $e<1/2$. However $\dt|q_P|\geq c\dt\eps^{-2/3}\gg |q_2\times p_2|$ as $\eps\to 0$, so $q_2\times p_2$ term is much smaller than $\dt q_P$.

Finally, we have that $O(|\tilde q_S|^2+|\tilde q_J|^2)=O(\dt^2|q_P|^2)$ using \eqref{EqTqJS} and $\tilde q=o(\dt)$, which is  much smaller than $\dt q_P$ due to $\dt=O(\eps^3)$. The $O(\dt^{3/2})$ from the first bullet point and $O(\dt\eps^{4/3})$ from the last bullet point are even smaller.

We have shown that the term $\dt \langle u,q_P\rangle$ is the leading term where $u$ is a constant vector denoting the sum of the integrals in the two brackets. Now we show $u\neq 0$ for $\mu>0$ sufficiently small, and the diameters of the two Lyapunov periodic orbits $\gm_1,\gm_2$ are also sufficiently small. From the definitions, we have $b_i=c_i+(1,0)$ and $c_i=O(\mu^{1/3})$ using Lemma \ref{Lm: 1/3}.
When $q_{Ai}=0$, the averaging procedure in the brackets is not needed. In the first bracket we perform Taylor expansion around $c_i=0$ to get $-2(c_1-c_2)+O(\mu^{2/3})$ and in the second bracket, we use Lemma \ref{Lm: 1/3} to substitute $c_i=\pm 3^{-1/3}\mu^{1/3}+o(\mu^{1/3})$ to get $(3(c_1-c_2)+o(\mu^{1/3}))/\mu$. So the second and third lines of \eqref{Eqfinal} add up to $[-(c_1-c_2)+o(\mu^{1/3})]\cdot \dt q_P+o(\dt).$ This shows $u\neq 0$ when $q_{Ai}=0$. When $|q_{Ai}|\neq 0$ but small enough, we also have $u\neq 0.$

The expression $\dt \langle u,q_P\rangle$ is nonconstant as a function of $\nu$, since $q_P$ is a perturbed Kepler ellipse in rotating coordinates. This completes the proof of the lemma.
\end{proof}
\begin{Rk}\label{RkFinal}
In the case of $\mu$ not small or the diameters of the Lyapunov orbits are not small, e.g. $\mu\simeq 10^{-3},\ J=3.03$ being the realistic value, the same proof applies except that we need to show $u\neq 0$ in some other ways. We can use numeric results about the Lyapunov orbits $\gamma_1,\gamma_2$ to do the integrals in \eqref{Eqfinal} and check $u\neq 0$, but we do not have a proof here.
\end{Rk}
\begin{proof}[proof of main Theorem \ref{Thm: main}]
To complete the proof, it remains to estimate the diffusion time. According to the proof of Lemma \ref{Lm: nondeg}, we use
$|\dfrac{d}{d\nu}\dt\langle u,q_P\rangle|=O(\dt\eps^{-2/3})$ to estimate $|\dfrac{d}{d\nu}\bar{f}_1(\nu)-\bar{f}_2(\nu)|$ for some nonzero constant vector $u$ denoting the integrals in \eqref{Eqfinal}. To get $O(1)$ energy growth in \eqref{eq: 3} as $\dt,\eps\to 0$, the time it takes is at most of order $O(1/(\dt\eps^{-2/3}\times \eps))=O(\dt^{-1}\eps^{-1/3})$.
\end{proof}
\appendix
\section{Two-body problem in Delaunay and polar coordinates}
\subsection{Delaunay coordinates}\label{AppendixDelaunay}
The purpose of this appendix is to give a brief introduction to the Delaunay coordinates used in the paper.
The materials could be found in \cite{8}. For two-body problem of the form\[H(P,Q)=\dfrac{|P|^2}{2m}-\dfrac{k}{|Q|},\quad (P,Q)\in \R^4,\]
we know it is integrable in the Liouville-Arnold sense when $H<0$. So we have the action-angle variables $(L,\ell, G,g)$ to write the Hamiltonian can be written as
\[H(L,\ell, G,g)=-\dfrac{mk^2}{2L^2},\quad (L,\ell, G,g)\in T^*\T^2.\]
The Hamiltonian equations are
\[\dot{L}=\dot{G}=\dot{g}=0,\quad \dot{\ell}=\dfrac{mk^2}{L^3}.\]
If we define the quantities:\\
$E$: energy, $M$: angular momentum, $e$: eccentricity, $a$: semimajor, $b$: semiminor, \\
then we have the following relations which endow the Delaunay coordinates the physical and geometrical meanings.
\[a=\frac{L^2}{mk}, \ b=\frac{LG}{mk},\ E=-\frac{k}{2a},\ M=G,\ e=\sqrt{1-\left(\frac{G}{L}\right)^2}.\]
Moreover, $g$ is the argument of periapsis and $\ell$ is called the mean anomaly.
We also have the Kepler's law $\dfrac{a^3}{T^2}=\dfrac{1}{(2\pi)^2}$ which relates the semimajor $a$ and the period $T$ of the ellipse.
For two-body problem, consider a body with position $(q_1, q_2)$ and momentum $(p_1,p_2)$. we have the following formulas
\[\begin{cases}
& q_1=a(\cos u-e),\\
& q_2=a\sqrt{1-e^2}\sin u,
\end{cases}
\quad\begin{cases}
& p_1=-\sqrt{mk}a^{-1/2}\dfrac{\sin u}{1-e \cos u},\\
& p_2=\sqrt{mk}a^{-1/2}\dfrac{\sqrt{1-e^2}\cos u}{1-e \cos u},
\end{cases}\]
where $u$ and $l$ are related by \begin{equation}\label{Equl}u-e\sin u=\ell.\end{equation}

Convert everything except $u$ into Delaunay, we have the following

\begin{equation}\begin{cases}
& q_1=(L^2/mk)\left(\cos u-\sqrt{1-\dfrac{G^2}{L^2}}\right),\\
& q_2=(LG/mk)\sin u.
\end{cases}
\quad
\begin{cases}
& p_1=-\dfrac{mk}{L}\dfrac{\sin u}{1-\sqrt{1-\frac{G^2}{L^2}} \cos u},\\
& p_2=\dfrac{mk}{L^2}\dfrac{G\cos u}{1-\sqrt{1-\frac{G^2}{L^2}}\cos u}.
\end{cases}\label{eq: delaunay}\end{equation}

Finally, we rotate the $(q_1,q_2)$ and $(p_1,p_2)$ using the matrix
$\left[\begin{array}{cc}
\cos g& -\sin g\\
\sin g& \cos g
\end{array} \right].
$
\subsection{Polar coordinates}\label{AppendixPolar}
We introduce the polar coordinates for $(P,Q)=(\dot x,\dot y, x,y)$ using the relation  (c.f. Chapter 7 of \cite{8})
\begin{equation}
\label{EqPolar}
x=r\cos\theta,\quad y=r\sin\theta,\quad m\dot x= R\cos\theta-\dfrac{\Theta}{r}\sin\theta,\quad m\dot y=R\sin\theta+\dfrac{\Theta}{r}\cos\theta.
\end{equation} Our symplectic form becomes
\[dP\wedge dQ=dR\wedge dr+d\Theta\wedge d\theta.\]
The Hamiltonian becomes
\begin{equation}H_2=\left[\frac{1}{2\al}\left(R^2+\dfrac{\Theta^2}{r^2}\right)-\dfrac{\al}{r}\right],\label{eq: H3}\end{equation}
\[\dot r=\dfrac{1}{\al}R,\quad \dot\theta=\dfrac{1}{\al}\dfrac{\Theta}{r^2},\quad \dot R=\dfrac{1}{\al}\dfrac{\Theta^2}{r^3}-\dfrac{\al}{r},\quad \dot\Theta=0.\]

We get $\Theta$ is a constant of motion and \[\ddot r=\dfrac{1}{\al}\dot R=\dfrac{1}{\al}\left(\dfrac{1}{\al}\dfrac{\Theta^2}{r^3}-\dfrac{\al}{r}\right)\] is only an equation of $r$, which can be solved explicitly. For more details of solving this equation, see Chapter 7.4.1. of \cite{8}.

We are only interested in the case of circular Kepler motion, i.e. $\dot r=0$, in which case $\dot\theta$ is also a constant following from the Hamiltonian equation. The constant $r=\dfrac{\Theta}{\al}$ is determined from the equation $\dot R=0$ and hence $\dot\theta=\dfrac{\Theta}{\al}.$ We choose $\Theta=\al$ to get the normalization $\dot \theta=1$, which in turn implies that $r=1$. If we are in a rotating coordinates such that $\dot\theta=0$, we get that $\dot R=\dot r=\dot\Theta=\dot\theta=0$ so that the circular motion is a critical point of the two-body problem in rotating coordinates.
\section*{Acknowledgement}
I would like to thank Prof. V. Kaloshin for the careful instruction, Dr. Yong Zheng and Joseph Galante for the  helpful discussions and many suggestions for revising the paper. I would also like to thank the anonymous referees for their valuable suggestions.


\begin{thebibliography}{DGO}
\def\bi#1{\bibitem[#1]{#1}}
\bi{1}A. Delshams, R. de la Llave, T.M. Seara, A geometric mechanism for diffusion in Hamiltonian systems overcoming the large gap problem: heuristics and rigorous verification on a model, \textit{Memoirs of the AMS No. 844,} 2006
\bi{2}V. Gelfreich, D. Turaev, Unbounded energy growth in Hamiltonian systems with a slowly varying parameter, - {\it Communications in Mathematical Physics}, 2008
\bi{3}de la Llave, orbits of unbounded energy in perturbations of geodesic flows by periodic potentials, a simple construction. preprint.
\bi{4}C. Conley, Low energy transit orbits in the restricted three-body problem', {\it SIAM J. Appl. Math.} {\bf 16}, 732V746. 1968
\bi{5}D. Wilczak, P. Zgliczynski. Heteroclinic Connections Between Periodic Orbits in Planar Restricted Circular Three-Body Problem-A computer Assisted Proof. {\it Communications in Mathematical Physics}, 2003
\bi{6}L. P. Shilnikov,  A. L. Shinikov, D. V. Turaev,  L.O.Chua: Methods of qualitative theory in nonlinear dynamics. Part I,{\it Singapore: World Scientific}, 1998
\bi{7}V. I. Arnold, V. V. Kozlov, A. Neishtadt, Mathematical aspects of classical and celestial mechanics.  Dynamical Systems {\bf III}, {\it Springer Verlag, New York}, 1988
\bi{8}K. R. Meyer, G. R. Hall, D. Offin, Introduction to Hamiltonian Dynamical systems and the N-Body Problem. 2nd edition, {\it Springer}. 2008
\bi{9} R.B. Barrar, Existence of periodic orbits of the second kind in the restricted problem of three bodies. \textit{The Astronomical Journal}, Vol \textbf{70}, No.\textbf{1}, 3-4, 1965
\bi{10}J. D. Hadjidemetrion, the continuation of periodic orbits from the restricted to the general three-body problem, {\it Celest. Mech.}, {\bf 12}. 155-174. 1975
\bi{11}P. Lochak, A. Neishtadt.  estimate stability time for nearly integrable systems with quasi convex hamiltonian. {\it Chaos}. {\bf 2(4)}:495-499, 1992 Oct;
\bi{12}W. S. Koon, M. W. Lo, J. E. Marsden, S. D. Ross, Heteroclinic connections between periodic orbits and resonance transitions in celestial mechanics, - {\it Chaos: An Interdisciplinary Journal of Nonlinear Science}, 2000
\bi{13}J. Llibre, R. Martinez, C. Simo:  Transversality of the invariant manifolds associated to the Lyapunov family of periodic orbits near L2 in the restricted three-body problem, {\it J. Diff. Eqns}, {\bf 58}, 104-156, 1985
\bi{14}Z. Xia, Arnold diffusion in the elliptic restricted three-body problem, {\it Journal of Dynamics and Differential Equations}, 1993
\bi{15} R. Moeckel, Transition tori in the five-body problem, {\it J. Diff. Equa.} \textbf{129}, 1996.
\bi{16} Y. Zheng, Arnold diffusion for a-priori unstable systems and a five-body problem, preprint
\bi{17} J. Fejoz, M. Guardia, V. Kaloshin, P. Roldan. Diffusion along mean motion resonance in the restricted planar three-body problem, preprint.
\bi{18} C. Sim\'{o}, T.J. Stuchi, Central stable/unstable manifolds and the destruction of KAM tori in the planar Hill problem, Physica D 140 (2000) 1-32.
\bi{19} A. Kiselev, V. Sverak, Small scale creation for solutions of the incompressible two dimensional Euler equation,  arXiv:1310.4799v2.

\end{thebibliography}
\end{document}